\documentclass[11pt,letterpaper]{article}

\usepackage[margin=1in]{geometry}
\usepackage[T1]{fontenc}
\usepackage{lmodern}
\usepackage{microtype}
\usepackage{amsmath,amssymb,amsthm,mathtools}
\usepackage{booktabs}
\usepackage{array}
\usepackage{enumitem}
\usepackage{needspace}
\usepackage{placeins}
\usepackage{xcolor}
\usepackage{tikz}
\usetikzlibrary{arrows.meta,calc,angles,quotes}
\usepackage[colorlinks=true,linkcolor=blue!55!black,urlcolor=blue!55!black]{hyperref}
\hypersetup{
  pdftitle={Straightedge-and-Compass Constructibility of the Reciprocal-k-th-Power Law},
  pdfauthor={George M. Georgiou}
}

\definecolor{constructionblue}{RGB}{30,92,150}
\definecolor{auxgray}{RGB}{120,125,130}
\definecolor{answerorange}{RGB}{205,105,25}

\newtheorem{theorem}{Theorem}
\newtheorem{proposition}[theorem]{Proposition}
\theoremstyle{remark}
\newtheorem{remark}[theorem]{Remark}

\title{\bfseries
Straightedge-and-Compass Constructibility\\[2pt]
of the Reciprocal-$k$-th-Power Law}
\author{\small George M. Georgiou\thanks{\small School of Computer Science and Engineering, California State University, San Bernardino}\\[3pt]
\small \href{mailto:georgiou@csusb.edu}{georgiou@csusb.edu}}
\date{}

\begin{document}
\maketitle
\vspace{-2.5em}

\begin{abstract}
Given positive lengths $a$ and $b$ and a positive integer $k$, let
$c$ be determined by
$$
  \frac{1}{c^k}=\frac{1}{a^k}+\frac{1}{b^k}.
$$
We prove that a single finite unmarked-straightedge-and-compass
construction producing $c$ for every pair $a,b$ exists if and only
if $k$ is a power of two.  Necessity follows by specializing to
$a=b=1$ and applying the algebraic-degree obstruction to
$2^{1/k}$; sufficiency is established by a recursive construction
using third, fourth, and mean proportionals.  We give an explicit
construction for $k=4$ and its iteration for $k=8,16,\ldots$.
We also describe a plane construction, intersecting a line with a
Lam\'e curve, that realizes the same relation geometrically for
every real $k>1$, although generally not by classical
straightedge-and-compass operations.
\end{abstract}


\section{The problem}

Let \(a\), \(b\), and \(c\) be positive lengths.  The reciprocal-\(k\)-th-power
identity \(1/c^{k}=1/a^{k}+1/b^{k}\) is equivalent to
\begin{equation}\label{eq:c-formula}
  c=\frac{ab}{(a^k+b^k)^{1/k}}.
\end{equation}
Here ``straightedge and compass'' has its classical meaning: the straightedge
is unmarked, and the only permitted intersections are those of lines and
circles already constructed.

There are two distinct questions:
\begin{enumerate}[label=(\roman*),leftmargin=2.4em]
  \item Does a geometric curve produce the length \(c\)?
  \item Can \(c\) be obtained from arbitrary given segments \(a\) and \(b\)
        by finitely many classical straightedge-and-compass operations?
\end{enumerate}
Intersection with a Lam\'e curve (Section~\ref{sec:lame})
answers the first question affirmatively for every real \(k>1\).  The second
question has a much more restrictive answer.

\section{The exact constructibility theorem}

\begin{theorem}[Complete answer for integer exponents]\label{thm:classification}
Let \(k\) be a positive integer.  There is a single finite
straightedge-and-compass construction that, from every pair of positive input
segments \(a,b\), produces a segment \(c\) satisfying
\[
  \frac1{c^k}=\frac1{a^k}+\frac1{b^k},
  \qquad\text{equivalently}\qquad
  c=\frac{ab}{(a^k+b^k)^{1/k}},
\]
if and only if \(k=2^m\) for some nonnegative integer \(m\).
\end{theorem}

\begin{proof}
First suppose that such a construction exists for every \(a,b\).  Apply it to
the two unit segments \(a=b=1\); the unit length is then present, and since the
reciprocal of a nonzero constructible length is constructible, the construction
would yield both
\[
  c=2^{-1/k}
  \qquad\text{and}\qquad
  \frac1c=2^{1/k}.
\]
The polynomial \(X^k-2\) is irreducible over \(\mathbb{Q}\) by Eisenstein's
criterion at the prime \(2\)~\cite[Cor.~14, pp.~309--310]{DummitFoote}.
Consequently \(2^{1/k}\) has algebraic degree \(k\) over \(\mathbb{Q}\).  The
standard field-theoretic criterion for straightedge-and-compass
constructibility implies that the degree of a constructible real number over
\(\mathbb{Q}\) is a power of two; this necessary condition goes back to
Wantzel~\cite[p.~369]{Wantzel}; see also~\cite[Sec.~13.3,
pp.~531--533]{DummitFoote}.  Therefore \(k\) must be a power of two.

Conversely, let \(k=2^m\).  The cases \(m=0\) and \(m=1\), namely \(k=1\) and
\(k=2\), are the familiar reciprocal-linear and altitude-to-the-hypotenuse
constructions; the latter is the reciprocal Pythagorean relation~\cite{Nelsen}.
For \(m\geq 2\), the recursive construction in
Section~\ref{sec:general} uses only third proportionals, right triangles,
mean proportionals, and a final fourth proportional.  All are standard
straightedge-and-compass constructions, and the calculation there proves that
the resulting segment has length exactly given by \eqref{eq:c-formula}.
\end{proof}

\subsection*{The obstruction already occurs at
  \texorpdfstring{\(k=3\)}{k=3}}

If \(k=3\) and \(a=b=1\), then
\[
  c=\frac{1}{\sqrt[3]{2}}.
\]
Constructing \(c\) would construct its reciprocal \(\sqrt[3]{2}\), contradicting
the classical impossibility of duplicating the
cube~\cite[p.~369]{Wantzel}; see also~\cite[Sec.~13.3, pp.~532--533]{DummitFoote}.
For an elementary account, see~\cite[Ch.~III, \S3, pp.~134--139]{CourantRobbins}.
Thus there cannot be a straightedge-and-compass construction valid for arbitrary
\(a\) and \(b\).
In particular, the absence of a construction for \(k=3\) is not a gap waiting
for a more ingenious diagram; it is an algebraic impossibility.

\begin{center}
\renewcommand{\arraystretch}{1.15}
\begin{tabular}{ccl}
\toprule
Exponent \(k\) & Classical construction? & Reason\\
\midrule
\(3\)          & No  & cube-root obstruction\\
\(4\)          & Yes & two successive square-root levels\\
\(5,6,7\)      & No  & degree not a power of two\\
\(8\)          & Yes & three successive square-root levels\\
\(9,\ldots,15\)& No  & degree not a power of two\\
\(16\)         & Yes & four successive square-root levels\\
\bottomrule
\end{tabular}
\end{center}

\begin{remark}
Theorem~\ref{thm:classification} concerns \emph{universal} constructibility: a
single procedure valid for \emph{every} pair \(a,b\).  It does not assert that
\(c\) is nonconstructible for each individual input.  When \(k\) is not a power
of two there may still be special pairs \(a,b\) for which \(c\) happens to be
constructible; the theorem says only that no one straightedge-and-compass
procedure can succeed for all inputs.  For example, at \(k=3\) take
\[
  a=\frac{3+\sqrt{93}}{7},\qquad b=1,\qquad
  c=\frac{\sqrt{93}-3}{7}:
\]
a short computation gives \(a^3-c^3=(ac)^3\), i.e. \(1/c^3=1/a^3+1/b^3\), and
all three lengths lie in the constructible field \(\mathbb{Q}(\sqrt{93})\),
even though \(3\) is not a power of two.  Such pairs arise as rational-slope
chords through the point \((0,1)\) on the Fermat cubic \(u^3-x^3=1\): the
residual intersection is governed by a \emph{quadratic}, whose roots are
constructible, and each admissible slope yields another example.
\end{remark}

\section{An explicit construction for
  \texorpdfstring{\(k=4\)}{k=4}}

The quartic case is the first new exponent after \(k=2\).  Starting with the
given segments \(a\) and \(b\), perform the following four standard construction
stages:
\begin{align}
  u&=\frac{a^2}{b},                                      \label{eq:k4-u}\\
  v&=\sqrt{u^2+b^2},                                     \label{eq:k4-v}\\
  w&=\sqrt{bv},                                          \label{eq:k4-w}\\
  c&=\frac{ab}{w}.                                       \label{eq:k4-c}
\end{align}

\Needspace*{10\baselineskip}
\begin{enumerate}[leftmargin=2.2em]
  \item Construct \(u\) as the third proportional to \(b,a\):
        \[
          b:a=a:u.
        \]
  \item Construct \(v\) as the hypotenuse of a right triangle with legs
        \(u\) and \(b\).
  \item Construct \(w\) as the mean proportional between \(b\) and \(v\).
  \item Construct \(c\) as the fourth proportional satisfying
        \[
          w:a=b:c.
        \]
\end{enumerate}

The porism to Proposition~VI.8 of Euclid's \emph{Elements} gives the
mean-proportional theorem used in Step~3.  Propositions~VI.11, VI.12, and
VI.13 construct, respectively, a third, fourth, and mean
proportional~\cite{Euclid}.

\begin{figure}[htbp]
\centering
\begin{tikzpicture}[
  scale=0.94,
  line cap=round,
  line join=round,
  >={Latex[length=2.2mm]},
  point/.style={circle,fill=black,inner sep=1.35pt},
  guide/.style={thin,auxgray},
  main/.style={very thick,constructionblue},
  parallel/.style={thick,answerorange}
]
% Panel 1: third proportional
\begin{scope}
  \coordinate (O) at (0,0);
  \coordinate (A) at (2.8,0);
  \coordinate (B) at (2.0,0);
  \coordinate (C) at (55:2.0);
  \coordinate (D) at (55:1.4286);
  \draw[guide,-{Latex[length=1.7mm]}] (O)--(3.35,0);
  \draw[guide,-{Latex[length=1.7mm]}] (O)--(55:2.55);
  \draw[main] (A)--(C);
  \draw[parallel] (B)--(D);
  \foreach \P in {O,A,B,C,D} \node[point] at (\P) {};
  \node[below left] at (O) {$O$};
  \node[above right,font=\footnotesize] at (A) {$A$};
  \node[above right,font=\footnotesize] at (B) {$B$};
  \node[above left] at (C) {$C$};
  \node[right,font=\footnotesize] at (D) {$D$};
  \draw[|<->|,thin] (0,-0.64)--(2.8,-0.64)
       node[midway,fill=white,inner sep=1pt] {$OA=b$};
  \draw[|<->|,thin] (0,-0.29)--(2.0,-0.29)
       node[midway,fill=white,inner sep=1pt,font=\footnotesize] {$OB=a$};
  \node[left,font=\footnotesize] at (55:1.80) {$OC=a$};
  \node[left,font=\footnotesize] at (55:0.85) {$OD=u$};
  \node[font=\bfseries] at (1.65,-1.28) {(a) Third proportional};
\end{scope}

% Panel 2: fourth proportional
\begin{scope}[xshift=7.3cm]
  \coordinate (O2) at (0,0);
  \coordinate (A2) at (2.9,0);
  \coordinate (B2) at (2.0,0);
  \coordinate (C2) at (55:2.7);
  \coordinate (D2) at (55:1.8621);
  \draw[guide,-{Latex[length=1.7mm]}] (O2)--(3.45,0);
  \draw[guide,-{Latex[length=1.7mm]}] (O2)--(55:3.05);
  \draw[main] (A2)--(C2);
  \draw[parallel] (B2)--(D2);
  \foreach \P in {O2,A2,B2,C2,D2} \node[point] at (\P) {};
  \node[below left] at (O2) {$O$};
  \node[above right,font=\footnotesize] at (A2) {$A$};
  \node[above right,font=\footnotesize] at (B2) {$B$};
  \node[above left] at (C2) {$C$};
  \node[right,font=\footnotesize] at (D2) {$D$};
  \draw[|<->|,thin] (0,-0.64)--(2.9,-0.64)
       node[midway,fill=white,inner sep=1pt] {$OA=w$};
  \draw[|<->|,thin] (0,-0.29)--(2.0,-0.29)
       node[midway,fill=white,inner sep=1pt,font=\footnotesize] {$OB=a$};
  \node[left,font=\footnotesize] at (55:2.05) {$OC=b$};
  \node[left,font=\footnotesize] at (55:1.05) {$OD=c$};
  \node[font=\bfseries] at (1.70,-1.28) {(b) Fourth proportional};
\end{scope}
\end{tikzpicture}
\caption{The two proportional constructions used for \(k=4\).
In (a), similarity of \(\triangle OBD\) and \(\triangle OAC\) gives
\(OD/OC=OB/OA=a/b\), and hence \(u=OD=a^2/b\).
In (b), the same construction with \(OA=w\), \(OB=a\), and \(OC=b\)
gives \(c=OD=ab/w\).  The displayed configurations use convenient sample
lengths; the order of points on a ray may change with the input lengths,
without changing the similarity argument.}
\label{fig:proportionals}
\end{figure}

\begin{figure}[htbp]
\centering
\begin{tikzpicture}[
  scale=0.95,
  line cap=round,
  line join=round,
  >={Latex[length=2.2mm]},
  point/.style={circle,fill=black,inner sep=1.35pt},
  guide/.style={thin,auxgray},
  main/.style={very thick,constructionblue},
  answer/.style={very thick,answerorange}
]
% Panel 1: right triangle
\begin{scope}
  \coordinate (O) at (0,0);
  \coordinate (A) at (1.65,0);
  \coordinate (B) at (1.65,2.65);
  \draw[main] (O)--(A)--(B)--cycle;
  \draw (1.65,0.25)--(1.40,0.25)--(1.40,0);
  \foreach \P in {O,A,B} \node[point] at (\P) {};
  \node[below left] at (O) {$O$};
  \node[below right] at (A) {$A$};
  \node[above right] at (B) {$B$};
  \node[below] at ($(O)!0.5!(A)$) {$u$};
  \node[right] at ($(A)!0.5!(B)$) {$b$};
  \node[left] at ($(O)!0.5!(B)$) {$v=\sqrt{u^2+b^2}$};
  \node[font=\bfseries] at (0.82,-0.78) {(a) Right triangle};
\end{scope}

% Panel 2: geometric mean
\begin{scope}[xshift=6.2cm]
  \def\rr{2.35}
  \def\vv{2.75}
  \pgfmathsetmacro{\tot}{\rr+\vv}
  \pgfmathsetmacro{\rad}{\tot/2}
  \pgfmathsetmacro{\cc}{sqrt(\rr*\vv)}
  \coordinate (A2) at (0,0);
  \coordinate (B2) at (\rr,0);
  \coordinate (C2) at (\tot,0);
  \coordinate (E2) at (\rr,\cc);
  \draw[main] (A2)--(C2);
  \draw[main] (A2) arc[start angle=180,end angle=0,radius=\rad];
  \draw[answer] (B2)--(E2);
  \draw (\rr,0.25)--(\rr-0.25,0.25)--(\rr-0.25,0);
  \foreach \P in {A2,B2,C2,E2} \node[point] at (\P) {};
  \node[below left] at (A2) {$A$};
  \node[below] at (B2) {$B$};
  \node[below right] at (C2) {$C$};
  \node[above] at (E2) {$E$};
  \node[below] at ($(A2)!0.5!(B2)$) {$b$};
  \node[below] at ($(B2)!0.5!(C2)$) {$v$};
  \node[right,answerorange] at ($(B2)!0.52!(E2)$)
       {$w=\sqrt{bv}$};
  \node[font=\bfseries] at (\rad,-0.78) {(b) Mean proportional};
\end{scope}
\end{tikzpicture}
\caption{The two metric operations in the quartic construction.
Panel (a) constructs \(v\) by the Pythagorean theorem.  In panel (b),
\(AC\) is the diameter of a semicircle, \(AB=b\), \(BC=v\), and
\(BE\perp AC\).  The altitude theorem gives \(BE^2=AB\cdot BC=bv\), so
\(BE=w=\sqrt{bv}\).}
\label{fig:metric}
\end{figure}

\FloatBarrier

\begin{proposition}[Verification of the quartic construction]
The segment \(c\) constructed by
\eqref{eq:k4-u}--\eqref{eq:k4-c} satisfies
\[
  \frac1{c^4}=\frac1{a^4}+\frac1{b^4}.
\]
\end{proposition}

\begin{proof}
The first two steps give
\[
  v
  =\sqrt{\frac{a^4}{b^2}+b^2}
  =\frac{\sqrt{a^4+b^4}}{b}.
\]
Consequently,
\[
  w=\sqrt{bv}=(a^4+b^4)^{1/4}.
\]
The last proportional construction therefore gives
\[
  c=\frac{ab}{w}
    =\frac{ab}{(a^4+b^4)^{1/4}},
\]
and hence
\[
  \frac1{c^4}
  =\frac{a^4+b^4}{a^4b^4}
  =\frac1{a^4}+\frac1{b^4}.
\]
\end{proof}

\section{The construction for every
  \texorpdfstring{\(k=2^m\)}{k=2 to the power m}}
\label{sec:general}

Let \(k=2^m\) with \(m\geq1\).  When \(m=1\), the two recursive families below
are empty, and the construction reduces to the familiar \(k=2\) case.  Starting
from \(x_0=a\), repeatedly construct third proportionals
\begin{equation}\label{eq:x-recursion}
  x_j=\frac{x_{j-1}^2}{b},
  \qquad j=1,\ldots,m-1.
\end{equation}
Inductively,
\[
  x_j=\frac{a^{2^j}}{b^{2^j-1}},
  \qquad
  x_{m-1}=\frac{a^{k/2}}{b^{k/2-1}}.
\]
Now construct the hypotenuse
\begin{equation}\label{eq:y0}
  y_0=\sqrt{x_{m-1}^2+b^2}
     =\frac{(a^k+b^k)^{1/2}}{b^{k/2-1}}.
\end{equation}
Next take \(m-1\) successive mean proportionals with \(b\):
\begin{equation}\label{eq:y-recursion}
  y_j=\sqrt{b\,y_{j-1}},
  \qquad j=1,\ldots,m-1.
\end{equation}
A second induction gives
\[
  y_j=
  \frac{(a^k+b^k)^{1/2^{j+1}}}
       {b^{\,k/2^{j+1}-1}}.
\]
At the last step, \(j=m-1\), this becomes
\[
  y_{m-1}=(a^k+b^k)^{1/k}.
\]
Finally, construct the fourth proportional
\[
  y_{m-1}:a=b:c.
\]
It produces
\[
  c=\frac{ab}{y_{m-1}}
   =\frac{ab}{(a^k+b^k)^{1/k}},
\]
which is precisely the required length.  In total the procedure uses \(m-1\)
third proportionals, one right triangle, \(m-1\) mean proportionals, and one
fourth proportional, that is, \(2m=2\log_2 k\) standard construction stages,
each consisting of a fixed finite number of elementary
straightedge-and-compass operations, so the number of standard stages grows
only logarithmically with \(k\).

\subsection*{Example:
  \texorpdfstring{\(k=8\)}{k=8}}

For \(k=8\), the construction reads
\[
\begin{aligned}
  x_1&=\frac{a^2}{b},&
  x_2&=\frac{x_1^2}{b}=\frac{a^4}{b^3},\\
  y_0&=\sqrt{x_2^2+b^2}
      =\frac{\sqrt{a^8+b^8}}{b^3},&
  y_1&=\sqrt{by_0}
      =\frac{(a^8+b^8)^{1/4}}{b},\\
  y_2&=\sqrt{by_1}
      =(a^8+b^8)^{1/8},&
  c&=\frac{ab}{y_2}.
\end{aligned}
\]
Thus only third proportionals, one right triangle, two mean proportionals,
and one fourth proportional are needed.

\section{Intersection with a Lam\'e curve}
\label{sec:lame}

Fix \(h>0\) and consider the Lam\'e arc
\begin{equation}\label{eq:lame}
  \left(\frac{X}{b}\right)^k+
  \left(1-\frac{Y}{h}\right)^k=1,
  \qquad 0\le X\le b,\quad 0\le Y\le h,
\end{equation}
(a translated arc of a Lam\'e curve; for \(k>2\), curves of this type were
popularized by Piet Hein under the name
\emph{superellipses}~\cite{Gridgeman,Gardner}), which runs from the origin
\(O\) to the point \((b,h)\).  Let \(C=(0,h)\), let \(A=(a,0)\), and let \(P\)
be the point where the segment \(AC\) meets the arc (Figure~\ref{fig:lame}).
That segment satisfies \(1-\tfrac{Y}{h}=\tfrac{X}{a}\); substituting into
\eqref{eq:lame} gives
\(\bigl(\tfrac{X}{b}\bigr)^{k}+\bigl(\tfrac{X}{a}\bigr)^{k}=1\).  Thus if
\(P=(X,Y)\), then
\[
  X=\frac{ab}{(a^k+b^k)^{1/k}}=c
\]
for every real \(k>1\), by~\eqref{eq:c-formula}.  The height \(h\) cancels, so
the abscissa of \(P\) is independent of it, and the perpendicular from \(P\) to
the horizontal axis cuts off exactly the segment \(c\).  This is a valid geometric
construction in the broad sense of intersecting a line with a prescribed curve.
It is not, however, a classical straightedge-and-compass construction unless the
required intersection point can itself be constructed using lines and circles.
Gardner is cited for background on the superellipse; we have not found the
particular intersection construction in that source.

\begin{figure}[htbp]
\centering
\begin{tikzpicture}[
  scale=1.15,
  line cap=round,
  line join=round,
  >={Latex[length=2.2mm]},
  point/.style={circle,fill=black,inner sep=1.3pt},
  guide/.style={thin,auxgray},
  surf/.style={very thick,constructionblue},
  answer/.style={very thick,answerorange}
]
  \def\Rr{2} \def\hh{3}
  \draw[guide,->] (-0.6,0)--(3.7,0) node[right]{$X$};
  \draw[guide,->] (0,-0.15)--(0,3.6) node[above]{$Y$};
  \draw[surf] plot[domain=0:1,samples=160,variable=\t]
      ({\Rr*(1-(1-\t)^3)^(1/3)},{\hh*\t});
  \coordinate (C)  at (0,\hh);
  \coordinate (Z)  at (3,0);
  \coordinate (P)  at (1.8343,1.1657);
  \coordinate (Gz) at (1.8343,0);
  \draw[guide] (Z)--(C);
  \draw[answer,densely dashed] (P)--(Gz);
  \node[point] at (0,0) {};
  \node[point] at (C) {};
  \node[point] at (Z) {};
  \node[point] at (P) {};
  \node[point] at (Gz) {};
  \node[below left] at (0,0) {$O$};
  \node[above left,font=\footnotesize] at (C) {$C$};
  \node[below,font=\footnotesize] at (Z) {$A$};
  \node[above right,font=\footnotesize] at (P) {$P$};
  \node[left,font=\footnotesize] at (0,1.55) {$h$};
  \draw[|<->|,thin] (0,-0.55)--(1.8343,-0.55)
       node[midway,fill=white,inner sep=1pt,font=\footnotesize] {$c$};
  \draw[|<->|,thin] (0,-1.08)--(3,-1.08)
       node[midway,fill=white,inner sep=1pt,font=\footnotesize] {$a$};
  \draw[|<->|,thin] (0,\hh+0.36)--(\Rr,\hh+0.36)
       node[midway,fill=white,inner sep=1pt,font=\footnotesize] {$b$};
\end{tikzpicture}
\caption{The Lam\'e arc \eqref{eq:lame}, drawn for \(k=3\).  The arc runs from
\(O\) to \((b,h)\), and \(C=(0,h)\).  The segment from \(A=(a,0)\) to \(C\)
meets the arc at \(P\); the perpendicular from \(P\) to the horizontal axis cuts
off \(c=ab/(a^k+b^k)^{1/k}\).  The height \(h\) cancels, so the abscissa of
\(P\) does not depend on it.  For general \(k\) the arc is neither a line nor a
circle, so this
need not be a classical construction.}
\label{fig:lame}
\end{figure}
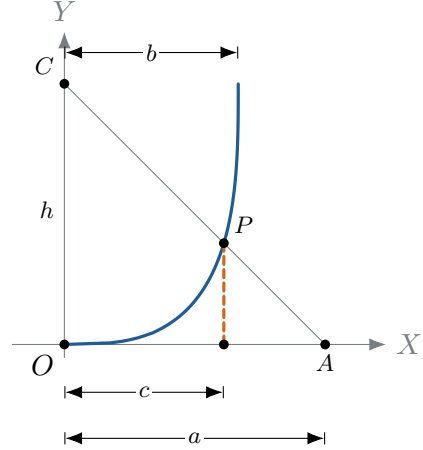

For example, when \(k=3\), taking \(a=b\) gives
\[
  c=\frac{b}{\sqrt[3]{2}},
  \qquad\text{so that}\qquad
  \frac{b}{c}=\sqrt[3]{2},
\]
and constructing \(c\) would construct \(\sqrt[3]{2}\), precisely the extra
nonclassical operation that straightedge and compass cannot perform.  When
\(k=4,8,16,\ldots\), the point \(P\) is classically constructible for every
constructible input, by the procedure above, even though a whole Lam\'e curve
is not itself drawn by a finite sequence of line-and-circle operations.

\begin{remark}
If a marked ruler (neusis) is admitted, cube roots (and hence the case
\(k=3\)) become constructible~\cite[Sec.~13.3, pp.~534--535]{DummitFoote}.
Suitable origami folds can, more generally, solve cubic
equations~\cite[pp.~120--121, 129--132]{Alperin}.
Such constructions are interesting, but they belong to a strictly larger class
than the classical unmarked-straightedge-and-compass constructions considered
here.
\end{remark}

\FloatBarrier
%\Needspace*{11\baselineskip}

\section*{Conclusion}

The passage from \(k=2\) to higher exponents does not yield a classical
construction for every \(k\).  Instead it reveals a sharp arithmetic
dichotomy:
\begin{center}
\setlength{\fboxsep}{6pt}
\fbox{\parbox{0.88\linewidth}{\centering
For positive integer \(k\), universal straightedge-and-compass
constructibility holds if and only if \(k\) is a power of \(2\).}}
\end{center}
The first universally constructible exponent beyond \(k=2\) is \(k=4\); no
universal construction exists for \(k=3\); and the quartic procedure iterates
naturally to \(k=8,16,\ldots\).

\end{document}